\documentclass[reqno,a4paper,12pt]{amsart} 

\usepackage{amsmath,amscd,amsfonts,amssymb}
\usepackage{mathrsfs,dsfont}

\usepackage{tikz}

\numberwithin{equation}{section}
\numberwithin{figure}{section}

\newcommand\R{\mathbb{R}}

\newcommand\Z{\mathbb{Z}}

\newcommand\gam{\gamma}
\newcommand\Gam{\Gamma}
\newcommand\lam{\lambda}
\newcommand\Lam{\Lambda}
\newcommand\Sig{\Sigma}

\newcommand\Om{\Omega}

\renewcommand\le{\leqslant}
\renewcommand\ge{\geqslant}
\renewcommand\leq{\leqslant}
\renewcommand\geq{\geqslant}
\newcommand\sbt{\subset}

\newcommand{\dotprod}[2]{\langle #1 , #2 \rangle}

\newcommand{\half}{\tfrac{1}{2}}

\theoremstyle{plain}
\newtheorem{thm}{Theorem}[section]
\newtheorem{lem}[thm]{Lemma}

\newtheorem{prop}[thm]{Proposition}

\newtheorem*{claim*}{Claim}

\newcommand{\thmref}[1]{Theorem~\ref{#1}}
\newcommand{\secref}[1]{Section~\ref{#1}}

\newcommand{\lemref}[1]{Lemma~\ref{#1}}
\newcommand{\defref}[1]{Definition~\ref{#1}}
\newcommand{\propref}[1]{Proposition~\ref{#1}}

\theoremstyle{definition}
\newtheorem{definition}[thm]{Definition}
\newtheorem*{definition*}{Definition}
\newtheorem*{remarks*}{Remarks}
\newtheorem*{remark*}{Remark}

\newenvironment{enumerate-roman}
{\begin{enumerate}
\addtolength{\itemsep}{5pt}
}
{\end{enumerate}}

\newenvironment{enumerate-alph}
{\begin{enumerate}
\addtolength{\itemsep}{5pt}
}
{\end{enumerate}}

\newenvironment{enumerate-num}
{\begin{enumerate}
\addtolength{\itemsep}{5pt}
}
{\end{enumerate}}

\newenvironment{enumerate-text}
{\begin{enumerate}
\addtolength{\itemsep}{5pt}
}
{\end{enumerate}}

\begin{document}

\title[Riesz bases of exponentials for convex polytopes]
{Riesz bases of exponentials for convex polytopes with symmetric faces}

\author{Alberto Debernardi}
\address{Department of Mathematics, Bar-Ilan University, Ramat-Gan 5290002, Israel}
\email{adebernardipinos@gmail.com}

\author{Nir Lev}
\address{Department of Mathematics, Bar-Ilan University, Ramat-Gan 5290002, Israel}
\email{levnir@math.biu.ac.il}

\date{January 26, 2020}
\thanks{Research supported by ISF Grants No.\ 447/16 and 227/17 and ERC Starting Grant No.\ 713927.}

\begin{abstract}
We prove that for any convex polytope $\Omega \subset \mathbb{R}^d$ which is centrally symmetric and whose faces of all dimensions are also centrally symmetric, there exists a Riesz basis of exponential functions in the space $L^2(\Omega)$. The result is new in all dimensions $d$ greater than one.
\end{abstract}

\maketitle


\section{Introduction} \label{secI1}

\subsection{}
Let $\Om \subset \R^d$ be a bounded, measurable set of positive 
measure. When is it possible to find  a countable set   $\Lam \sbt
\R^d$ such that the system of exponential functions
\begin{equation}
	\label{eqI1.1}
	E(\Lambda)=\{e_\lambda\}_{\lambda\in \Lambda}, \quad e_\lambda(x)=e^{2\pi
	i\dotprod{\lambda}{x}},
\end{equation}
constitute a basis in the space $L^2(\Om)$?

The answer depends on what we mean by a ``basis''. 
The best one can hope for is to have an \emph{orthogonal basis}
of exponentials.
The problem of which domains admit an orthogonal basis $E(\Lam)$
has been extensively studied and goes back to
Fuglede \cite{Fug74}, and it is well-known that  many
reasonable domains  do not have such a basis.
For example, it was recently proved in \cite{LM19} that 
a convex domain $\Om \sbt \R^d$ admits
an orthogonal basis of exponential functions  if and only if $\Om$
is a convex polytope which can tile the space 
by translations (meaning that one can fill the whole space by
translated copies of $\Om$ with non-overlapping interiors).
In particular, a disk or a triangle in the plane does not have
an orthogonal basis $E(\Lam)$ (this was shown
already in \cite{Fug74}).

If a domain $\Om \sbt \R^d$
 does not have an orthogonal basis of exponentials,
then a \emph{Riesz basis} is the next best thing one can hope for.
A Riesz basis can be defined as the image of an orthonormal basis 
under a bounded and invertible linear map (there are also other,
 equivalent definitions, see 
 \secref{sect:prelim}) and it shares many of the qualities of
an orthonormal basis. In particular, if $E(\Lam)$ is a Riesz basis
 in $L^2(\Om)$ then any function $f$ from the space
has a unique and stable Fourier series expansion 
$f = \sum_{\lam \in \Lam} c_\lam e_\lam$.

In sharp contrast to the situation with orthogonal bases,
no single example is known of a set $\Om \sbt \R^d$ 
which does not have a Riesz basis of exponentials.
At the same time, for most domains it remains unknown
whether one can construct a Riesz basis of this type.
One of the relatively few results obtained in this direction says that 
any finite union of intervals in $\R$ has a Riesz 
basis of exponentials \cite{KN15} (see also \cite{KN16}
for a multi-dimensional version of this result).
However it is still an open problem of whether, say,
a disk or a triangle in the plane admits a Riesz basis $E(\Lam)$.

\subsection{}
The present paper is concerned with the existence of Riesz bases of 
exponentials for \emph{convex polytopes} in $\R^d$. Our main result can
be stated as follows:

\begin{thm}
\label{thmA1}
Let $\Om \sbt \R^d$ be a convex polytope which is centrally
symmetric and all of whose faces 
of all dimensions are also centrally symmetric. 
Then there is a set $\Lam \sbt \R^d$ such that
the system of exponential functions $E(\Lam)$
is a Riesz basis in $L^2(\Om)$.
\end{thm}

The result is new in all dimensions $d$ greater than one.

In \cite{LR00}, Lyubarskii and Rashkovskii established the
existence of a Riesz basis of exponentials  for
convex, centrally symmetric polygons in $\R^2$ such that
all the vertices of the polygon lie on the integer lattice
$\Z^2$ (one may alternatively assume that the 
vertices lie on some other lattice, due to the
invariance of the problem under affine 
transformations).\footnote{The assumption that all the 
vertices of the polygon lie on a lattice was stated in 
\cite{LR00} in a different (equivalent) form, by imposing 
a system of arithmetic constraints given in \cite[Proposition 3.1(v)]{LR00}.}
The approach in \cite{LR00} involves methods from the 
theory of entire functions of two complex variables.
The paper \cite{LR00} contains also a weaker result for
convex, centrally symmetric polygons whose vertices
fail to lie on a lattice, but in this case the result does
not amount to the construction of a Riesz basis of exponentials.

A similar result in higher dimensions was obtained in
 \cite[Corollary 3]{GL14}, where the
existence of a Riesz basis of exponentials 
was established for  centrally 
symmetric polytopes in $\R^d$ with 
centrally symmetric facets, such that all the  vertices
of the polytope lie on the lattice $\Z^d$.
The proof is based on the fact that such a polytope
\emph{multi-tiles} the space by lattice translates
(in connection with this result, 
 see also \cite{Kol15}, \cite{GL18}).

In this paper, our goal is to prove the existence of a Riesz 
basis of exponentials for convex, centrally
symmetric polytopes with centrally symmetric faces,
\emph{without imposing any extra constraints}.
This is the content of our main result, \thmref{thmA1}.

\subsection{}
Our approach to the proof of \thmref{thmA1} is inspired by the paper
\cite{Wal17} due to Walnut. In that paper, the author
applies a technique outlined in \cite{BFW06} in order
to construct a system of exponentials $E(\Lam)$
that is shown to be \emph{complete} in the space $L^2(\Om)$,
where $\Om \subset \R^2$ is  a convex, 
centrally symmetric polygon. The set $\Lambda$ 
constructed in \cite{Wal17} is the union of a 
finite number of shifted lattices in $\R^2$. 
It is shown that if the convex polygon satisfies  certain 
extra arithmetic constraints given in \cite[Theorem 4.2]{Wal17}, 
then $E(\Lambda)$ is not only a 
complete system, but is in fact a Riesz basis, in $L^2(\Om)$.

In \cite{Wal17} the author does not provide any transparent 
description as to which convex, centrally symmetric polygons 
satisfy the extra constraints imposed in \cite[Theorem 4.2]{Wal17}. 
One can verify though
that these constraints are satisfied if and only if,
possibly after applying an affine transformation,
all the  vertices of the polygon lie in $\Z^2$.
Hence the class of planar convex polygons for which a 
Riesz basis $E(\Lam)$ is constructed in \cite{Wal17}
coincides  with the class covered by the result in \cite{LR00}.

In this paper we will extend the technique from
\cite{Wal17} to all dimensions, and also combine it with the 
Paley-Wiener theorem about the stability of Riesz bases under small perturbations (see \secref{sect:prelim}).
This will allow us to  construct a Riesz basis $E(\Lam)$ 
for any convex, centrally
symmetric polytope in $\R^d$ with centrally symmetric faces,
without imposing any extra arithmetic constraints.
The set of frequencies $\Lam$ in our construction 
will  no longer be a union of finitely many shifted lattices, 
but it will rather have a less regular structure.


\section{Preliminaries}
\label{sect:prelim}

\subsection{Riesz bases}
Let $H$ be a separable Hilbert space. 
A system of vectors $\{f_n\} \subset H$
is called a \emph{Riesz basis} if it is 
the image of an orthonormal basis 
under a bounded and invertible linear map.
If $\{f_n\}$ is a Riesz basis then any
$f \in H$ admits a unique expansion in a series 
$f = \sum c_n f_n$, and the coefficients
$\{c_n\}$ satisfy
$A \|f\|^2 \leq \sum |c_n|^2 \leq B\|f\|^2$
for some positive constants $A, B$ which
do not depend on $f$. In fact, the latter
condition for the system $\{f_n\}$
can serve as an  equivalent definition
of a Riesz basis.

There are also several other ways to characterize 
Riesz bases in a separable Hilbert space $H$. The following
characterization, see
\cite[Section 4.4, Theorem 8]{You01}, will be used
in the present paper:

\begin{prop}
\label{propP1.6}
A system $\{f_n\} \subset H$ is a Riesz basis if and only if
it satisfies the following three conditions:
\begin{enumerate-num}
\item \label{rb:i} $\{f_n\}$ is a complete system in $H$;
\item \label{rb:ii}  for every $f \in H$ we have
$\sum_n \left| \dotprod{f}{f_n} \right|^2 < \infty$;
\item \label{rb:iii} given any 
sequence of scalars $\{c_n\}$ such that
$\sum_n |c_n|^2 < \infty$,
there exists at least one $f \in H$ satisfying
$\dotprod{f}{f_n} = c_n$
for all $n$.
\end{enumerate-num}
\end{prop}

For a discussion about the various properties and
characterizations of Riesz bases we refer the
reader to \cite{You01}.

\subsection{Paley-Wiener spaces}
Let $\Om \sbt \R^d$ be a bounded, measurable set of
positive measure. The Paley-Wiener space $PW(\Om)$
consists of all functions $F \in L^2(\R^d)$ which are
Fourier transforms of functions from $L^2(\Om)$, 
namely,
\begin{equation}
F(t) = \int_{\Om} f(x) \, e^{-2\pi i \dotprod{t}{x}}\, dx, \quad f \in L^2(\Om).
\end{equation}

A set  $\Lam \sbt \R^d$ is called
a \emph{set of uniqueness} for the space $PW(\Om)$
if whenever a function $F$ from the space satisfies
$F(\lam)=0$, $\lam \in \Lam$, then $F$ is identically
zero. This means that the functions from the
space $PW(\Om)$ are  uniquely determined by their
values on $\Lam$. 

We say that $\Lam$ is a \emph{set of interpolation} for  
$PW(\Om)$ if for any 
$\{c(\lam)\} \in \ell^2(\Lam)$ there exists
at least one 
$F \in PW(\Om)$ satisfying $F(\lam) = c(\lam)$,
$\lam \in \Lam$.
Such a function $F$ is said to solve the interpolation
problem with the set of nodes $\Lam$ and with the values
$\{c(\lam)\}$.

The Fourier transform is a unitary map from the space
 $L^2(\Om)$ onto $PW(\Om)$. This allows to reformulate
the uniqueness and interpolation properties of a set $\Lambda \sbt \R^d$ 
with respect to the space $PW(\Om)$, in terms
of properties  of the system of exponential functions $E(\Lam)$
in the space $L^2(\Om)$. Thus
$\Lam$ is a set of uniqueness for  $PW(\Om)$
if and only if  $E(\Lam)$ is a complete system  in 
$L^2(\Om)$; while $\Lam$ is a set of interpolation for  $PW(\Om)$
if and only if the system of equations
$\dotprod{f}{e_\lam} = c(\lam)$, $\lam \in \Lam$, admits at
least one solution $f \in L^2(\Om)$ whenever the scalars
$\{c(\lam)\}$ belong to $\ell^2(\Lam)$.

\subsection{Uniformly discrete sets}
A set $\Lambda\subset \R^d$ is said to be \emph{uniformly discrete}
 if there is $\delta>0$ such that $|\lambda'-\lambda|\ge \delta$ 
for any two distinct points $\lambda,\lambda'$ in $\Lambda$. 
The following  proposition may be found for instance
 in \cite[Proposition 2.7]{OU16}.

\begin{prop}
\label{propP1.2}
Let $\Om \subset \R^d$ be a bounded, measurable set of positive 
measure, and let $\Lam \sbt \R^d$ be a uniformly
discrete set. Then there is a constant $C = C(\Om,\Lam)$
such that 
\[
\sum_{\lam \in \Lam} |F(\lam)|^2 \leq C \|F\|^2_{L^2(\R^d)}
\]
for every function $F \in PW(\Om)$.
\end{prop}

It is well-known that if 
$\Om \sbt \R^d$ is a bounded, measurable set of
positive measure, and if $\Lam \sbt \R^d$ is a 
set of interpolation for the space $PW(\Om)$,
 then $\Lam$ must be a uniformly
discrete set, see e.g.\ \cite[Section 4.2.1]{OU16}.
Due to Propositions \ref{propP1.6} and 
 \ref{propP1.2}, this implies the
following characterization of Riesz bases of
exponentials in the space $L^2(\Om)$:

\begin{prop}
\label{propP1.5}
Let $\Om \subset \R^d$ be a bounded, measurable set of positive 
measure, and let $\Lam \sbt \R^d$. The system
of exponentials $E(\Lam)$ is a Riesz basis
in $L^2(\Om)$ if and only if $\Lam$ is a set of
both uniqueness and interpolation for the space
$PW(\Om)$. 
\end{prop}

\subsection{Stability}
We will need the following result which goes back to
Paley and Wiener \cite{PW34} about the stability of Riesz 
bases of exponentials under small perturbations:

\begin{prop}
\label{propP1.9}
Let $\Om \subset \R^d$ be a bounded, measurable set of positive 
measure, and let $\Lam = \{ \lam_n \}$ be a sequence of points
in $\R^d$ such that the system $E(\Lambda)$ is a Riesz basis
in $L^2(\Om)$. Then there is a constant $\eta = \eta(\Om,\Lam) >0$
such that if a sequence $\Lam' = \{\lam'_n\}$ satisfies
$|\lam'_n - \lam_n| \leq \eta$
for all $n$, then also $E(\Lam')$ is a Riesz basis in $L^2(\Om)$.
\end{prop}

For a proof of this result the reader may consult
\cite[Section 4.3]{OU16}.

\subsection{Convex polytopes}
A set $\Om \sbt \R^d$ is called a \emph{convex polytope}
if $\Om$ is the convex hull of a finite number
of points. Equivalently,  a convex polytope is a bounded set 
which can be represented as the intersection
of finitely many closed halfspaces.

A convex polytope $\Om \subset \R^d$ is said to be
 \emph{centrally symmetric} if the set $-\Om$ is a translate
of $\Om$. In this case, there exists a unique point $x \in \R^d$ 
such that 
$- \Om + x = \Om - x$, and we  say that $\Om$ is  symmetric with respect to the point $x$.

A \emph{zonotope} in $\R^d$  is a set $\Om$
 which can be represented as the
Minkowski sum of a finite number of line segments, that is,
$\Om = S_1 + S_2 + \dots + S_n$ where each one of the
sets $S_1, S_2, \dots, S_n$ is a line segment in $\R^d$.
A zonotope in $\R^d$   can  be equivalently defined as
 the image of a cube in $\R^n$
under an affine map from $\R^n$ to $\R^d$.

 A zonotope is a convex, centrally symmetric polytope, 
and all its faces of all dimensions are also zonotopes. 
In particular, all  the faces of a zonotope are  centrally 
symmetric. The following proposition
shows that the  converse is also true:

\begin{prop}[{see \cite[Theorem 3.5.2]{Sch14}}]
\label{propP1.12}
Let $\Om$ be a convex polytope in $\R^d$.
Then the following conditions are equivalent:
\begin{enumerate-num}
\item \label{zo:i} $\Om$ is a zonotope 
(i.e.\ $\Om$ is the Minkowski
sum of  finitely many line segments);
\item \label{zo:ii}   $\Om$ is
centrally symmetric and all its faces 
of all dimensions are also centrally symmetric;
\item \label{zo:iii} all the two-dimensional
faces of $\Om$ are centrally symmetric.
\end{enumerate-num}
\end{prop}

Thus, for example, any convex
centrally symmetric polygon $\Om \subset \R^2$
 is a zonotope.


\section{Cylindric sets}
\label{sect:cylindric}

\subsection{}
We denote a point in $\R^d = \R^{d-1} \times \R$ as $(x,y)$
where $x \in \R^{d-1}$ and $y \in \R$.

\begin{definition}
\label{defC1.1}
A bounded, measurable set $\Sig \subset \R^d$ will be called a \emph{cylindric set} if there exists a bounded, measurable set 
$\Pi \sbt \R^{d-1}$ and a bounded, measurable function 
$\varphi: \Pi \to \R$ such that
\begin{equation}
	\label{eqC1.1}
	\Sig = \left\{(x,y) : x \in \Pi, \; \varphi(x)  \le y \le \varphi(x) + 1 \right\}.
\end{equation}
\end{definition}

In other words, $\Sig$ is a cylindric set if $\Sig \subset \Pi \times \R$
and  for every $x \in \Pi$, the set
$\{y : (x,y) \in \Sig\}$ is an interval of length exactly $1$,
where the position of this interval is allowed to depend on $x$.
The set $\Pi$ will be called the \emph{base} of the
cylindric set $\Sig$. 

In the special case when the function $\varphi$ in \eqref{eqC1.1}
 is constant, the cylindric set $\Sig$ has a cartesian product 
structure, namely $\Sig = \Pi \times I$ where $I$ is an interval 
of length 1. In this case it is well-known 
that if there is a set $\Gam \subset \R^{d-1}$ such that 
the system $E(\Gam)$ is a Riesz basis in $L^2(\Pi)$, then 
the system
$E(\Gam \times \Z)$ is a Riesz basis in $L^2(\Sig)$.

We will use the fact that the latter assertion remains valid
for arbitrary cylindric sets of the form \eqref{eqC1.1},
 not only for those with
 a cartesian product  structure:

\begin{lem}
\label{lemC1.1}
Let $\Sig \sbt \R^d$ be a cylindric set with base $\Pi$, and 
suppose that there is a set
$\Gam \subset \R^{d-1}$ such that the system of exponentials
$E(\Gam)$ is a Riesz basis in $L^2(\Pi)$. Then the system
$E(\Gam \times \Z)$ is a Riesz basis in $L^2(\Sig)$.
\end{lem}

The proof of this lemma for arbitrary cylindric sets
is  similar to the one for  cartesian products, which
should be well-known.
We were not able to find the proof in the literature though,
 so we include it below  for completeness.

\subsection{}
For the proof of \lemref{lemC1.1} we will use the following
characterization of the exponential systems $E(\Lam)$ that
form a Riesz basis in the space $L^2(\Om)$.

\begin{prop}
\label{propAP1.1}
Let $\Om \subset \R^d$ be a bounded, measurable set of positive 
measure, and let $\Lam \sbt \R^d$. The system
of exponentials $E(\Lam)$ is a Riesz basis
in $L^2(\Om)$ if and only if there exists a constant $M$ such that the
following two conditions hold:

\begin{enumerate-num}
\item \label{AP1.1.1} 
$\sum_\lam |c(\lam)|^2 \leq M \| \sum_{\lambda} c(\lambda) e_\lam \|^2_{L^2(\Om)}$
whenever $\{c(\lam)\}$, $\lam \in \Lam$, is a sequence of scalars with only
finitely many nonzero elements;
\item \label{AP1.1.2}
$\|f\|^2_{L^2(\Om)} \leq M \sum_{\lam} | \dotprod{f}{e_\lam}|^2$
for any $f \in L^2(\Om)$.
\end{enumerate-num}
\end{prop}

We note that condition \ref{AP1.1.1} holds if and only if $\Lam$
is a set of interpolation for the space $PW(\Om)$, 
see \cite[Sections 4.1, 4.2]{OU16}. It follows from
\ref{AP1.1.1} that $\Lam$ must be  a uniformly
discrete set.

If the condition \ref{AP1.1.2} holds, then $\Lam$ is said to be a
\emph{set of stable sampling} for the space $PW(\Om)$, see
\cite[Section 2.5]{OU16}. This condition  implies
in particular that $\Lam$ is a  set of uniqueness  for the space.

\subsection{}
\emph{Proof of \lemref{lemC1.1}}.
We suppose that $\Sig \sbt \R^d$ is a cylindric set of the form
\eqref{eqC1.1},  and that $\Gam \subset \R^{d-1}$ is a set
such that the system of exponentials
$E(\Gam)$ is a Riesz basis in $L^2(\Pi)$. We must prove that
the system $E(\Gam \times \Z)$ is a Riesz basis in $L^2(\Sig)$.
We will show that the two conditions \ref{AP1.1.1} and
\ref{AP1.1.2} in \propref{propAP1.1} are satisfied.

We will use $\gamma$ to denote an element of $\Gam$, 
and $n$ to denote an element of $\Z$.

First we check that \ref{AP1.1.1} holds. Let $\{c(\gam,n)\}$ be 
a sequence of scalars with only finitely many nonzero elements.
We must show that $\sum_{\gam,n}
 |c(\gam,n)|^2 \leq M \|f\|^2_{L^2(\Sig)}$, where
\[
f(x,y) := \sum_{\gam,n} c(\gam,n) e_\gam(x) e_n(y).
\]
Let $\psi_n(x) := \sum_{\gam} c(\gam,n) e_\gam(x)$, 
and $I(x) := [\varphi(x), \varphi(x)+1]$. Then we have
\[
\label{AP1.7}
\int_{I(x)} |f(x,y)|^2 \, dy = \int_{I(x)}
\Big| \sum_{n} \psi_n(x) e_n(y) \Big|^2 dy 
=  \sum_{n} |\psi_n(x) |^2
\]
since $E(\Z)$ is an orthonormal basis in $L^2(I(x))$
for every $x$. In turn, this implies
\[
\|f\|^2_{L^2(\Sig)} = \int_{\Pi} \int_{I(x)}
|f(x,y)|^2 dy \, dx = \sum_{n} \| \psi_n \|^2_{L^2(\Pi)}
\geq \frac1{M} \sum_{n}  \sum_{\gam} |c(\gam,n)|^2,
\]
where the last inequality holds for a certain constant
$M = M(\Pi, \Gam)$ since the system 
$E(\Gam)$ is a Riesz basis in $L^2(\Pi)$. This
confirms that condition \ref{AP1.1.1} is indeed satisfied.

Next we check that \ref{AP1.1.2} holds. Let $f \in L^2(\Sig)$, then we have
\[
\|f\|^2_{L^2(\Sig)} = \int_{\Pi} \int_{I(x)} | f(x,y)|^2  \, dy \, dx
= \int_{\Pi} \Big( \sum_n |\phi_n(x)|^2 \Big) dx,
\]
where
\[
\phi_n(x) := \int_{I(x)} f(x,y) \, \overline{e_n(y)} \, dy.
\]
This is due to the fact that $E(\Z)$ is an orthonormal basis in $L^2(I(x))$. 
It follows that
\begin{equation}
\label{AP1.9}
\|f\|^2_{L^2(\Sig)}  = \sum_n \|\phi_n\|^2_{L^2(\Pi)}  \leq 
M \sum_{n} \sum_{\gam} | \dotprod{\phi_n}{e_{\gam}}|^2
= M\sum_{n} \sum_{\gam} | \dotprod{f}{e_{\gam,n}}|^2,
\end{equation}
where we denote $e_{\gam,n}(x,y) := e_\gam(x) e_n(y)$.
Observe that the inequality in \eqref{AP1.9} holds since 
$E(\Gam)$ is a Riesz basis in $L^2(\Pi)$.
This establishes \ref{AP1.1.2} and so the lemma is proved.
\qed


\section{Decomposition of functions with a zonotope spectrum}
\label{sect:decompose}

\subsection{}
Suppose that we are given $n$ vectors 
 $u_1, u_2, \dots, u_n$ in $\R^d$. 
The origin-symmetric zonotope generated by
these vectors is the set
\begin{equation}
	\label{eqC3.1.1}
	\Om_n = \Big\{ \sum_{j=1}^{n} t_j u_j \, : \, t_1,t_2,\dots,t_n \in [-\half, \half] \Big\}.
\end{equation}
This set is the Minkowski sum of the line segments
$[-\half u_j, \half u_j]$, $1 \leq j \leq n$,
and so it is indeed a zonotope in $\R^d$.

We observe that if the linear span of the 
vectors $u_1, u_2, \dots, u_n$ is the whole $\R^d$,
then the zonotope $\Om_n$ has nonempty interior; while
if these vectors do not span the whole $\R^d$
then $\Om_n$ is contained in some hyperplane,
and it is then a set of measure zero.

\begin{lem}
\label{lemC3.1}
Assume that the first $n-1$ vectors $u_1, u_2, \dots, u_{n-1}$
span the whole $\R^d$,  and that we have $u_n = (0,0,\dots,0,1)$. 
Then there is a
cylindric set $\Sig_n \sbt \R^d$  whose base $\Pi_n$
 is a zonotope in $\R^{d-1}$,  such that the following holds:

\begin{enumerate-num}
\item \label{dc:i}
 Any function $F \in PW(\Om_n)$ 
can be represented in the form
\begin{equation}
	\label{eqC3.1.2}
	F(x,y) = G(x,y) +  H(x,y) \sin (\pi y), 
	\quad (x,y) \in \R^{d-1} \times \R,
\end{equation}
for some $G \in PW(\Sig_n)$ and some $H \in PW(\Om_{n-1})$.
Here we denote by $\Om_{n-1}$ the origin-symmetric 
zonotope in $\R^d$ generated by the vectors
$u_1, u_2, \dots, u_{n-1}$.
\item \label{dc:ii}
 Conversely, for any two functions
$G \in PW(\Sig_n)$ and  $H \in PW(\Om_{n-1})$,
the function $F$ defined by \eqref{eqC3.1.2}
belongs to the space $PW(\Om_n)$.
\end{enumerate-num}
\end{lem}

This is an extension to all dimensions of
\cite[Lemma 3.1]{Wal17} where the result was proved
for a convex, centrally symmetric polygon $\Om_n$ 
in two dimensions. In the two-dimensional case,
 the set $\Sig_n$ was taken to be a parallelogram
that shares with $\Om_n$ its two edges 
parallel to the vector $u_n$.
In dimensions greater than two, we cannot
in general take  the set $\Sig_n$ 
in \lemref{lemC3.1} to be a parallelepiped, nor
any other type of convex polytope 
inscribed in $\Om_n$ in a similar way. Instead, 
the role of the parallelogram will be played 
in higher dimensions 
by the cylindric sets introduced in \defref{defC1.1}.

The assumption in \lemref{lemC3.1} 
that the first $n-1$ 
vectors $u_1, u_2, \dots, u_{n-1}$ span the whole $\R^d$
is made so as to ensure that 
the zonotope $\Om_{n-1}$ has nonempty interior.

\subsection{}
\emph{Proof of \lemref{lemC3.1}}.
By identifying the elements $F$, $G$ and $H$ of the spaces
 $PW(\Om_n)$, $PW(\Sig_n)$ and $PW(\Om_{n-1})$
with the Fourier transforms of functions $f$, $g$ and $h$ from $L^2(\Om_n)$, 
$L^2(\Sig_n)$ and $L^2(\Om_{n-1})$
respectively, the conditions \ref{dc:i} and \ref{dc:ii}
can be reformulated as follows: any function $f \in L^2(\Om_n)$ 
admits a representation of the form
\begin{equation}
	\label{eqC3.4.1}
	f(x,y) = g(x,y) + 
	\frac{h(x,y+\half) - h(x,y-\half)}{2i}
	\quad \text{a.e.} \quad 
	(x,y) \in \R^{d-1} \times \R
\end{equation}
for some $g \in L^2(\Sig_n)$ and some $h \in L^2(\Om_{n-1})$;
and conversely, for any two functions
$g \in L^2(\Sig_n)$ and $h \in L^2(\Om_{n-1})$,
the function $f$ defined by \eqref{eqC3.4.1}
belongs to $L^2(\Om_n)$.

(Notice that  we think of a function from the space
$L^2(\Om_n)$, $L^2(\Sig_n)$ or $L^2(\Om_{n-1})$
as a function on the whole space $\R^d$ which is 
assumed to vanish a.e.\ outside the set $\Om_n$,
$\Sig_n$ or $\Om_{n-1}$ respectively.)

Let $\Pi_n \subset \R^{d-1}$ be the image of $\Om_{n-1}$ 
under the map $(x,y) \mapsto x$. If we denote
$u_j = (v_j, w_j)$ where $v_j \in \R^{d-1}$
and $w_j \in \R$, then $\Pi_n$ is the
origin-symmetric zonotope generated by
the vectors $v_1, v_2, \dots, v_{n-1}$.
These vectors span the whole space
$\R^{d-1}$ and hence the
zonotope $\Pi_n$ has nonempty interior.

For each $x \in \Pi_n$ we let
\begin{equation}
	\label{eqC3.4.2}
	S(x) := \{y \in \R : (x,y) \in \Om_{n-1}\}.
\end{equation}
By the definition of $\Pi_n$ the set $S(x)$
is nonempty. Since $\Om_{n-1}$ is closed and convex, the set
$S(x)$ is a closed interval. We may therefore denote
$S(x) := [a(x), b(x)]$. 
Since $\Om_{n-1}$ is a convex polytope, each one of
$a(x)$, $b(x)$ is a continuous, piecewise linear 
function on $\Pi_n$. (This can be
deduced from the representation of $\Om_{n-1}$ as
the intersection of  finitely many closed
halfspaces.)

The zonotope $\Om_n$ is the Minkowski sum of
$\Om_{n-1}$ and the line segment $[-\half u_n, \half u_n]$.
 Since we have assumed that 
$u_n = (0,0,\dots,0,1)$ this implies that
\begin{equation}
	\label{eqC3.4.7}
	\Om_n = \left\{(x,y) : x \in \Pi_n, \; a(x) - \half
	 \le y \le b(x) + \half \right\},
\end{equation}
and in particular we have
\begin{equation}
	\label{eqC3.4.9}
	(\Om_{n-1} - \half u_n) \cup
	(\Om_{n-1} + \half u_n) \subset \Om_n.
\end{equation}

Let   $\varphi: \Pi_n \to \R$ be any bounded,
measurable function satisfying
\begin{equation}
	\label{eqC3.4.3}
a(x) \leq \varphi(x) \leq b(x),
\quad x \in \Pi_n.
\end{equation}
(For example, one may 
take $\varphi(x) := a(x)$.)
We define a cylindric set $\Sig_n$ by
\begin{equation}
	\label{eqC3.4.5}
	\Sig_n := \left\{(x,y) : x \in \Pi_n, \; \varphi(x) - \half
  \le y \le \varphi(x) + \half \right\}.
\end{equation}
 It follows from  \eqref{eqC3.4.7}, \eqref{eqC3.4.3} and \eqref{eqC3.4.5}
 that $\Sig_n$ is a subset of $\Om_n$ (see Figure \ref{fig:sign}).


\begin{figure}[t]
\centering
\begin{tikzpicture}[scale=0.375]

\path 
	(15,6) coordinate (S)
	arc [radius=8, start angle=40, end angle=90] 
	coordinate (P)
	arc [radius=8, start angle=-90, end angle=-140] 
	coordinate (T);

\fill [fill=gray!35]
	(15,4) coordinate (A) -- ++(0,4) coordinate (B)
	arc [radius=8, start angle=40, end angle=90] 
	arc [radius=8, start angle=-90, end angle=-140] 
	coordinate (C) -- ++(0, -4) coordinate (D)
	arc [radius=8, start angle=-140, end angle=-90]
	arc [radius=8, start angle=90, end angle=40];

\draw[help lines]
	(S) -- ++(-3,-2) -- ++(-4, 1) -- (T)
	 -- ++(3,2) -- ++(4,-1) -- (S);

\draw[help lines, -stealth]  (0, 0 |- P) -- (18, 0 |- P);
\draw[help lines, -stealth]  (P |- 0, 0) -- (P |- 0, 18);
\draw (18, 0 |- P) node[anchor=west] {$\R^{d-1}$};
\draw (P |- 0, 18) node[anchor=south] {$\R$};

\draw [loosely dashed] (B) 
	arc [radius=8, start angle=40, end angle=90] 
	arc [radius=8, start angle=-90, end angle=-140];
\draw [loosely dashed] (D) 
	arc [radius=8, start angle=-140, end angle=-90]
	arc [radius=8, start angle=90, end angle=40];

\draw[densely dotted, black!75!white] (A) -- (A |- 0,0);
\draw[densely dotted, black!75!white] (D) -- (D |- 0,0);

\draw
	(A) -- ++(-3,-2) -- ++(-4, 1) -- (D)
	-- (C) -- ++(3,2) -- ++(4,-1) -- (B) -- (A);

\draw[stealth-stealth,  line width=0.5mm]  (D |- 0,0) -- (A |- 0,0);
\draw (P |- 0,-1) node {\small $\Pi_n$};

\draw (12.25,13.75) node {\small $\Om_n$};
\draw[help lines] (16.75,6) node {\small $\Om_{n-1}$};
\draw (T) node[xshift=-0.5cm] {\small $\Sig_n$};

\end{tikzpicture}
	\caption{The various sets involved in the statement and 
		proof of \lemref{lemC3.1} are illustrated.
		The larger polygon represents $\Om_n$, 
		the smaller polygon is $\Om_{n-1}$,
		while the shaded region is the cylindric set $\Sig_n$.}
	\label{fig:sign}
\end{figure}
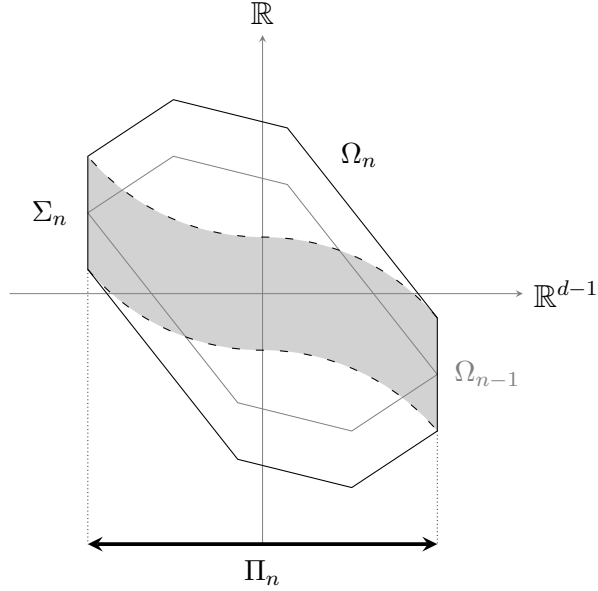


Suppose now that two functions
$g \in L^2(\Sig_n)$ and $h \in L^2(\Om_{n-1})$
are given, and let $f$ be  the function defined by \eqref{eqC3.4.1}.
Then $f$ is supported on the set
\[
\Sig_n \cup
	(\Om_{n-1} - \half u_n) \cup
	(\Om_{n-1} + \half u_n)
\]
which is contained in $\Om_n$. This shows that $f$
belongs to $L^2(\Om_n)$.

Conversely, suppose that we are given a function
$f \in L^2(\Om_n)$. We will show that $f$
admits a representation of the form
\eqref{eqC3.4.1}
where $g \in L^2(\Sig_n)$ and $h \in L^2(\Om_{n-1})$.

First we define the function $h$ by
\begin{gather}
	\label{eqC3.2.1}
	h(x,y) :=  2i \sum_{k \geq 0} f(x,y-k -\half), 
	\quad x \in \Pi_n, \quad y < \varphi(x),\\[4pt]
	\label{eqC3.2.2}
	h(x,y) := - 2i \sum_{k \geq 0} f(x,y+k +\half), 
	\quad x \in \Pi_n, \quad y > \varphi(x),
\end{gather}
and $h(x,y) := 0$ if $x \notin \Pi_n$.
(Notice that it is not necessary to define the values of $h$ on
the set of points $(x,y)$ such that $x \in \Pi_n$,
$y= \varphi(x)$, as this is a set of measure zero.)

We observe that there is a constant $M$ such that 
the nonzero terms in the sum in either
\eqref{eqC3.2.1} or  \eqref{eqC3.2.2} 
correspond only to values of $k$ that are
not greater than $M$. This is due to the
assumption that  $f$ is supported on $\Om_n$.
Indeed, one can check using  \eqref{eqC3.4.7} 
and \eqref{eqC3.4.3} that
it suffices to take $M := \max \, (b(x)-a(x))$,
$x \in \Pi_n$. This shows that $h$ is a well-defined
function in $L^2(\R^d)$.  It also follows from
\eqref{eqC3.4.7} that $h(x,y) = 0$ whenever
$x \in \Pi_n$ and $y \notin [a(x), b(x)]$.
We conclude that $h$ is supported on
$\Om_{n-1}$, so  $h \in L^2(\Om_{n-1})$.

Next we define the function $g$ by
\begin{equation}
	\label{eqC3.4.12}
	g(x,y) := f(x,y)  - 
	\frac{h(x,y+\half) - h(x,y-\half)}{2i},
	\quad 
	(x,y) \in \R^{d-1} \times \R.
\end{equation}
Then $g$ is a function in $L^2(\R^d)$ and 
\eqref{eqC3.4.1} is satisfied. It remains
only to show that 
\begin{equation}
	\label{eqC3.4.15}
	g(x,y) = 0
	\quad \text{a.e.} \quad
	(x,y) \in \R^d \setminus \Sig_n.
\end{equation}
It will be enough to verify this for $x \in \Pi_n$. Since
$(x,y) \notin \Sig_n$ we have two possibilities,
either $y < \varphi(x) - \half$ or $y >\varphi(x) + \half$.
In the former case, we obtain from \eqref{eqC3.2.1} that
\[
\frac{h(x,y+\half) - h(x,y-\half)}{2i} =
 \sum_{k \geq 0} f(x,y-k) -  \sum_{k \geq 0} f(x,y-k-1)
= f(x,y) \quad \text{a.e.,}
\]
while in the latter case, \eqref{eqC3.2.2} implies that
\[
\frac{h(x,y+\half) - h(x,y-\half)}{2i} =
- \sum_{k \geq 0} f(x,y+k+1) +  \sum_{k \geq 0} f(x,y+k)
= f(x,y) \quad \text{a.e.} 
\]
The condition \eqref{eqC3.4.15} is therefore established.
We have thus constructed the desired representation for 
the function $f$, and this completes the proof of \lemref{lemC3.1}.
\qed


\section{Construction of Riesz bases for zonotopes}
\label{sect:proof}

In this section we prove our main result, \thmref{thmA1}. 
The theorem asserts that if
$\Om \sbt \R^d$ is a convex, centrally
symmetric  polytope and all its faces 
of all dimensions are also centrally symmetric,
then there is a set $\Lam \sbt \R^d$ such that
the system of exponentials  $E(\Lam)$
is a Riesz basis in $L^2(\Om)$.
By \propref{propP1.12}, a convex polytope 
$\Om \sbt \R^d$ satisfies the assumptions
in \thmref{thmA1} if and only if it is a zonotope.
Hence \thmref{thmA1} can be equivalently stated in 
 the following way:

\begin{thm}
\label{thmD1}
Let $\Om_n \sbt \R^d$ be an origin-symmetric zonotope 
of the form \eqref{eqC3.1.1} which is generated by $n$ vectors 
 $u_1, u_2, \dots, u_n$ that span the whole $\R^d$.
Then there is a set $\Lam_n \sbt \R^d$ such that
the system  $E(\Lam_n)$
is a Riesz basis in $L^2(\Om_n)$.
\end{thm}

\begin{proof}
We may  assume that no two of the 
vectors  $u_1, u_2, \dots, u_n$  are collinear. 
We will prove the assertion by induction  both on the 
dimension $d$ and on the number $n$ of the vectors 
generating the zonotope $\Om_n$. 

The induction base case is when $n=d$ and then $\Om_n$ is 
generated by $d$ linearly independent vectors in $\R^d$.
In this case $\Om_n$ is a $d$-dimensional
parallelepiped and so we know that it
admits a Riesz basis (in fact, an orthogonal basis) of exponentials.

Suppose now that $n > d$, and therefore the
dimension $d$ must be at least two and 
the vectors  $u_1, u_2, \dots, u_n$ 
are not linearly independent. In this case  we may reorder
them so that the 
 first $n-1$ vectors  $u_1, u_2, \dots, u_{n-1}$ 
already span the whole space $\R^d$. We may also assume,
by applying an invertible linear map, that
the last vector $u_n = (0,0,\dots,0,1)$. 

We now observe that all
the assumptions in \lemref{lemC3.1} are satisfied.
It  thus follows from the lemma that there exists a
cylindric set $\Sig_n \sbt \R^d$   whose base $\Pi_n$
 is a zonotope in $\R^{d-1}$,
such that any function $F \in PW(\Om_n)$ 
can be represented in the form \eqref{eqC3.1.2}
where $G \in PW(\Sig_n)$ and $H \in PW(\Om_{n-1})$;
and conversely, given any two functions
$G \in PW(\Sig_n)$ and $H \in PW(\Om_{n-1})$,
the function $F$ defined by \eqref{eqC3.1.2}
belongs to  $PW(\Om_n)$.

The base $\Pi_n$ of the cylindric set $\Sig_n$ is a zonotope in
 $\R^{d-1}$ with nonempty interior,
 hence by the inductive hypothesis
there is a set $\Gam_{n} \sbt \R^{d-1}$ such that
the system
$E(\Gam_n)$ is a Riesz basis in the space
$L^2(\Pi_n)$. By \lemref{lemC1.1} 
the system  $E(\Gam_n \times \Z)$ is then
a Riesz basis in $L^2(\Sig_n)$.

The zonotope $\Om_{n-1}$ is generated by the $n-1$
vectors $u_1, u_2, \dots, u_{n-1}$ in $\R^d$ 
whose linear span is the whole $\R^d$. Hence,
again by the inductive hypothesis, there is a 
set $\Lam_{n-1} \sbt \R^d$ such that the system
$E(\Lam_{n-1})$ is a Riesz basis in $L^2(\Om_{n-1})$.

We now invoke the Paley-Wiener stability result given
in \propref{propP1.9}. The result says that there
is a constant $\eta = \eta(\Om_{n-1}, \Lam_{n-1})>0$
such that if $\Lam'_{n-1} \sbt \R^d$ is any set obtained from
$\Lam_{n-1}$ by perturbing each element by 
 distance at most $\eta$, then the system
$E(\Lam'_{n-1})$ is also a Riesz basis in $L^2(\Om_{n-1})$.

Recall that we denote a point in $\R^d = \R^{d-1} \times \R$ 
as $(x,y)$
where $x \in \R^{d-1}$ and $y \in \R$. Let a 
mapping $\phi: \R^d \to \R^d$ be defined by the
requirement that if $\phi(x,y) = (x',y')$, then
$x'=x$ and $y'$ is a point closest to $y$ in the set
\begin{equation}
\label{eqD2.8}
\R \setminus \bigcup_{k \in \Z} (k-\eta, k+\eta)
\end{equation}
(notice that if $y$ is an integer then there 
are two possible choices for $y'$).
It is obvious that such a mapping $\phi$ exists whenever
$0 < \eta \leq \half$, which we may assume by taking $\eta$
smaller if necessary. 
We then consider the set $\Lam'_{n-1} := \phi(\Lam_{n-1})$
defined to be the image of $\Lam_{n-1}$ under the mapping $\phi$.
Then each element of $\Lam'_{n-1}$ is obtained 
by perturbing an element of the set $\Lam_{n-1}$ by 
 distance at most $\eta$. Hence the system
$E(\Lam'_{n-1})$ is a Riesz basis in $L^2(\Om_{n-1})$
by  \propref{propP1.9}.

We now define
\begin{equation}
	\label{eqD2.3}
\Lam_n := (\Gam_n \times \Z) \cup \Lam'_{n-1}
\end{equation}
and claim that $E(\Lam_n)$ is a Riesz basis in the space $L^2(\Om_n)$.
The proof of this claim will follow the technique from
\cite{Wal17}. By \propref{propP1.5}
it would be enough if we  show that $\Lam_n$ is  a set of
both uniqueness and interpolation for the space $PW(\Om_n)$. 

We start with the uniqueness part of the proof.
 Let $F \in PW(\Om_n)$ be a
function such that $F(\lam) = 0$ for all $\lam \in
\Lam_n$. The function $F$ has a representation
  in the form \eqref{eqC3.1.2}
where $G \in PW(\Sig_n)$ and $H \in PW(\Om_{n-1})$.
It follows from \eqref{eqC3.1.2}
 that $F(x,y)=G(x,y)$ for every $(x,y) \in \R^{d-1} 
\times \Z$, hence using \eqref{eqD2.3} this implies
that $G$ vanishes on the
set $\Gam_n \times \Z$. But the system  $E(\Gam_n \times \Z)$ is
a Riesz basis in $L^2(\Sig_n)$, so in particular
$\Gam_n \times \Z$ is a 
set of uniqueness for $PW(\Sig_n)$. We conclude that
$G$ must vanish identically.
The expression \eqref{eqC3.1.2} thus becomes
\begin{equation}
	\label{eqD2.5}
	F(x,y) =   H(x,y) \sin (\pi y), 
	\quad (x,y) \in \R^{d-1} \times \R.
\end{equation}
Again using \eqref{eqD2.3} this now implies that
$H(x,y) \sin (\pi y) = 0$ for every
$(x,y) \in \Lam'_{n-1}$. But if $(x,y)$ is
a point in $\Lam'_{n-1}$ then $y$ lies in
the set \eqref{eqD2.8} and hence
$\sin (\pi y) \neq 0$, and it follows that
$H(x,y) = 0$. Since $E(\Lam'_{n-1})$ is
a Riesz basis in $L^2(\Om_{n-1})$ then
in particular $\Lam'_{n-1}$ is a
set of uniqueness for $PW(\Om_{n-1})$,
hence $H$ must also vanish identically. We conclude
that $F$ is identically zero, and $\Lam_n$ is a set of
uniqueness for $PW(\Om_n)$. 

We next turn to the interpolation part of the proof. 
Let $\{c(\lam)\}$ be scalar values in $\ell^2(\Lam_n)$,
and we must show that there is 
$F \in PW(\Om_n)$ satisfying $F(\lam) = c(\lam)$
 for every $\lam \in \Lam_n$. We will find the
solution $F$ based on the representation \eqref{eqC3.1.2}.
First we find a function $G \in PW(\Sig_n)$ such that
$G(\lam) = c(\lam)$ for every $\lam \in \Gam_n \times \Z$.
This is possible since $E(\Gam_n \times \Z)$ is
a Riesz basis in $L^2(\Sig_n)$, and hence
$\Gam_n \times \Z$ is a set of interpolation for the 
space $PW(\Sig_n)$. Now consider the system
of values
\begin{equation}
	\label{eqD2.9}
	\left\{
	\frac{c(x,y) - G(x,y)}{\sin (\pi y)} \; : \;
	(x,y) \in \Lam'_{n-1} \right\}.
\end{equation}
We claim that these values are in $\ell^2(\Lam'_{n-1})$.
Indeed, observe that by \propref{propP1.2}  the
values of  the function $G$ on the uniformly  discrete 
set $\Lam'_{n-1}$ are in $\ell^2(\Lam'_{n-1})$.
Furthermore, if $(x,y)$ is a point in $\Lam'_{n-1}$ then $y$ lies in
the set \eqref{eqD2.8}, and  we therefore have
\begin{equation}
\label{eqD2.10}
\inf_{(x,y) \in \Lam'_{n-1}} |\sin (\pi y)| >0.
\end{equation}
It follows from these properties that the values in 
\eqref{eqD2.9} indeed belong to  $\ell^2(\Lam'_{n-1})$.
Since  $E(\Lam'_{n-1})$ is a Riesz basis in $L^2(\Om_{n-1})$, 
then $\Lam'_{n-1}$  is a set of interpolation for
 $PW(\Om_{n-1})$  
and hence there is $H \in PW(\Om_{n-1})$ such that
\begin{equation}
	\label{eqD2.12}
	H(x,y) = 
	\frac{c(x,y) - G(x,y)}{\sin (\pi y)}, \quad
	(x,y) \in \Lam'_{n-1}.
\end{equation}
Using the fact that $\sin (\pi y) = 0$ for every
$(x,y) \in \Gam_n \times \Z$, we  can now
conclude that the function  $F$ defined by \eqref{eqC3.1.2}
solves the interpolation problem 
$F(\lam) = c(\lam)$, $\lam \in \Lam_n$,
and this function belongs to  $PW(\Om_n)$.
We thus obtain that the set $\Lam_n$ is also a set of 
interpolation for $PW(\Om_n)$. This completes the proof of
\thmref{thmD1}.
\end{proof}


\end{document}